\documentclass[10pt]{amsart}
\usepackage{amsfonts, amssymb, amsmath, amsthm, euscript, color}
\usepackage{geometry}
\usepackage{graphicx}
\usepackage{amssymb}
\usepackage{epstopdf}
\usepackage{url}
\usepackage[all]{xy}

\title{Connected Hopf algebras of dimension $p^2$}
\author{Xingting Wang}
\address{
Xingting Wang\\ Department of Mathematics\\ University of Washington\\ Seattle, WA 98195}
\email{xingting@uw.edu}

\thanks{Research of the author was partially supported by the US National Science Foundation.}


\date{March, 2013}                                           

\newcommand*{\e}{\ensuremath{\varepsilon}}

\newcommand*{\field}{\ensuremath{\bold{k}}}
\newcommand*{\Prim}{\ensuremath{\text{\upshape P}}}

\newcommand*{\Hom}{\ensuremath{\text{\upshape Hom}}}

\newcommand*{\gr}{\ensuremath{\text{\upshape gr}}}
\newcommand*{\HL}{\ensuremath{\text{\upshape H}}}
\newcommand*{\Ext}{\ensuremath{\text{\upshape Ext}}}
\newcommand*{\Ker}{\ensuremath{\text{\upshape Ker}}}

\newcommand*{\ad}{\ensuremath{\text{\upshape ad}}}
\newcommand*{\End}{\ensuremath{\text{\upshape End}}}
\newcommand*{\Img}{\ensuremath{\text{\upshape Im}}}

\newcommand*{\degree}{\ensuremath{\text{\upshape d}}}
\newcommand*{\PBW}{\ensuremath{\text{\upshape PBW}}}
\newcommand*{\Lie}{\ensuremath{\text{\upshape Lie}}}

\newtheorem*{theorem1*}{Theorem \ref{generalgrc}}
\newtheorem*{theorem2*}{Theorem \ref{FREENESS}}
\newtheorem*{theorem3*}{Theorem \ref{NPLA}}
\newtheorem*{theorem4*}{Theorem  \ref{centerP1}}
\newtheorem*{theorem5*}{Theorem  \ref{D2}}
\newtheorem*{theorem6*}{Theorem \ref{NPLA}}
\newtheorem{theorem}{Theorem}[section]
\newtheorem{lemma}[theorem]{Lemma}
\newtheorem{corollary}[theorem]{Corollary}
\newtheorem{proposition}[theorem]{Proposition}

\theoremstyle{definition}
\newtheorem{remark}[theorem]{Remark}
\newtheorem{definition}[theorem]{Definition}

\newtheorem{convention}[theorem]{Convention}

\begin{document}
\begin{abstract}
Let $H$ be a finite-dimensional connected Hopf algebra over an algebraically closed field $\field$ of characteristic $p>0$. We provide the algebra structure of the associated graded Hopf algebra $\gr H$. Then, we study the case when $H$ is generated by a Hopf subalgebra $K$ and another element and the case when $H$ is cocommutative. When $H$ is a restricted universal enveloping algebra, we give a specific basis for the second term of the Hochschild cohomology of the coalgebra $H$ with coefficients in the trivial $H$-bicomodule $\field$. Finally, we classify all connected Hopf algebras of dimension $p^2$ over $\field$.  
\end{abstract}

\maketitle

\section{Introduction}
Let $\field$ denote a base field, algebraically closed of characteristic $p>0$. In \cite{Hen}, all graded cocommutative connected Hopf algebras of dimension less than or equal to $p^3$ are classified by using W.M.~Singer's theory of extensions of connected Hopf algebras \cite{WMS}. In this paper, we classify all connected Hopf algebras of dimension $p^2$ over $\field$. We use the theories of restricted Lie algebras and Hochschild cohomology of coalgebras for restricted universal enveloping algebras.

Let $H$ denote a finite-dimensional connected Hopf algebra in the sense of \cite[Def. 5.1.5]{MO93} with primitive space $\Prim(H)$, and $K$ be a Hopf subalgebra of $H$. In Section 2, basic definitions related to and properties of $H$ are briefly reviewed. In particular, we describe a few concepts concerning the inclusion $K\subseteq H$. We say that the \emph{$p$-index} of $K$ in $H$ is $n-m$ if $\dim K=p^m$ and $\dim H=p^n$. The notion of the \emph{first order} of the inclusion and a \emph{level-one} inclusion are also given in Definition \ref{BDCHA}.

In Section 3, the algebra structure of a finite-dimensional connected coradically graded Hopf algebra is obtained (Theorem \ref{generalgrc}) based on a result for algebras representing finite connected group schemes over $\field$. It implies that the associated graded Hopf algebra $\gr H$ is isomorphic to as algebras \[\field\left[x_1,x_2,\cdots,x_d\right]/\left(x_1^p,x_2^p,\cdots,x_d^p\right)\] for some $d\ge0$.

Section 4 concerns a simple case when $H$ is generated by $K$ and another element $x$. Suppose the $p$-index of $K$ in $H$ is $d$. Under an additional assumption, the basis of $H$ as a left $K$-module is given in terms of the powers of $x$ (Theorem \ref{FREENESS}). Moreover, if $K$ is normal in $H$ \cite[Def. 3.4.1]{MO93}, then $x$ satisfies a polynomial equation as follows:
\begin{align*}
x^{p^d}+\sum_{i=0}^{d-1}a_ix^{p^i}+b=0
\end{align*}
for some $a_i\in \field$ and $b\in K$. 

Section 5 deals with the special case when $H$ is cocommutative. It is proved in Proposition \ref{cocommfiltration} that such Hopf algebra $H$ is equipped with a series of normal Hopf subalgebras $\field=N_0\subset N_1\subset N_2\subset \cdots \subset N_n=H$ satisfying certain properties. If we apply these properties to the case when $\Prim(H)$ is one-dimensional, then we have $N_1$ is generated by $\Prim(H)$ and each $N_i$ has $p$-index one in $N_{i+1}$ (Corollary \ref{FCLH}). In Theorem \ref{NPLA}, we give locality criterion for $H$ in terms of its primitive elements. This result, after dualization, is equivalent to a criteria for unipotency of finite connected group schemes over $\field$, as shown in Remark \ref{GSL}.

In section 6, we take the Hopf subalgebra $K=u\left(\mathfrak g\right)$, the restricted universal enveloping algebra of some finite-dimensional restricted Lie algebra $\mathfrak g$. We consider the Hochschild cohomology of the coalgebra $K$ with coefficients in the trivial bicomodule $\field$, namely $\HL^\bullet(\field,K)$. Then the Hochschild cohomology can be computed as the homology of the cobar construction of $K$. In Proposition \ref{Liealgebrainclusion}, we give a specific basis for $\HL^2(\field,K)$. We further show, in Lemma \ref{chainmap}, that $\bigoplus_{n\ge 0} \HL^n(\field ,K)$ is a graded restricted $\mathfrak g$-module via the adjoint map. When the inclusion $K\subseteq H$ has first order $n\ge 2$, the differential $d^1$ in the cobar construction of $H$ induces a restricted $\mathfrak g$-module map from $H_n$ into $\HL^2(\field, K)$, whose kernel is $K_n$ (Theorem \ref{Cohomologylemma}). Concluded in Theorem \ref{HCT}, if $K\neq H$, we can find some $x\in H\setminus K$ with the following comultiplication   
\[
\Delta(x)=x\otimes 1+1\otimes x+\omega\left(\sum_i\alpha_ix_i\right)+\sum_{j<k}\alpha_{jk}x_j\otimes x_k
\]
where $\{x_i\}$ is a basis for $\mathfrak g$.

Finally, the classification of connected Hopf algebras of dimension $p^2$ over $\field$ is accomplished in section 7. Assume $\dim H=p^2$. We apply results on $H$ from previous sections, i.e., Corollary \ref{FCLH} and Theorem \ref{HCT}. The main result is stated in Theorem \ref{D2} and divided into two cases. When $\dim \Prim(H)=2$, based on the classification of two-dimensional Lie algebras with restricted maps (see Appendix A), there are five non-isomorphic classes
\begin{itemize}
\item[(1)] $\field\left[x,y\right]/\left(x^p,y^p\right)$,
\item[(2)] $\field\left[x,y\right]/\left(x^p-x,y^p\right)$,
\item[(3)] $\field\left[x,y\right]/\left(x^p-y,y^p\right)$,
\item[(4)] $\field\left[x,y\right]/\left(x^p-x,y^p-y\right)$,
\item[(5)] $\field\langle x,y\rangle/\left([x,y]-y,x^p-x,y^p\right)$,
\end{itemize}
where $x,y$ are primitive. When $\dim\Prim(H)=1$, $H$ must be commutative and there are three non-isomorphic classes
\begin{itemize}
\item[(6)] $\field\left[x,y\right]/(x^p,y^p)$,
\item[(7)] $\field\left[x,y\right]/(x^p,y^p-x)$,
\item[(8)] $\field\left[x,y\right]/(x^p-x,y^p-y)$,
\end{itemize}
where $\Delta\left(x\right)=x\otimes 1+1\otimes x$ and $\Delta\left(y\right)=y\otimes 1+1\otimes y+\omega(x)$.
Moreover, all local Hopf algebras of dimension $p^2$ over $\field$ are classified by duality, see Corollary \ref{localp2}.

\section{Preliminaries}
Throughout this paper, $\field$ denotes a base field, algebraically closed of characteristic $p>0$. All vector spaces, algebras, coalgebras, and tensor products are taken over $\field$ unless otherwise stated. Also, $V^*$ denotes the vector space dual of any vector space $V$. 

For any coalgebra $C$, the \bf{coradical}\rm\ $C_0$ is defined to be the sum of all simple subcoalgebras of $C$. Following \cite[5.2.1]{MO93}, $\{C_n\}_{n=0}^\infty$ is used to denote the \bf{coradical filtration}\rm\ of $C$. If $C_0$ is one-dimensional, $C$ is called \textbf{connected}. If every simple subcoalgebra of $C$ is one-dimensional, $C$ is called \textbf{pointed}. Let $(C,\Delta,\e)$ be a pointed coalgebra, and $\left(M,\rho_l,\rho_r\right)$ be a $C$-bicomodule via the structure maps $\rho_l: M\to C\otimes M$ and $\rho_r: M\to M\otimes C$. We denote the identity map of $C^{\otimes n}$ by $I_n$ and $C^{\otimes 0}=\field$. The \textbf{Hochschild cohomology} $\HL^\bullet\left(M,C\right)$ of $C$ with coefficients in $M$ is defined by the homology of the complex $\left(\mathbb C^n(M,C),d^n\right)$, where $\mathbb C^n(M,C)=\Hom_\field\left(M,C^{\otimes n}\right)$ and 
\begin{align*}
d^n(f)=(I\otimes f)\rho_l-(\Delta\otimes I_{n-1})f+\cdots+(-1)^n(I_{n-1}\otimes \Delta)f+(-1)^{n+1}(f\otimes I)\rho_r.
\end{align*}

For any Hopf algebra $H$, we use $\Prim(H)$ to indicate the subspace of primitive elements. Following the terminology in \cite[Def. 1.13]{Andruskiewitsch02pointedhopf}, we recall the definition of graded Hopf algebras.
\begin{definition}
Let $H$ be a Hopf algebra with antipode $S$. If
\begin{itemize}
\item[(1)] $H=\bigoplus_{n=0}^{\infty}H(n)$ is a graded algebra,
\item[(2)] $H=\bigoplus_{n=0}^{\infty}H(n)$ is a graded coalgebra,
\item[(3)] $S(H(n))\subseteq H(n)$ for any $n\ge 0$,
\end{itemize}
then $H$ is called a \bf{graded Hopf algebra}\rm.\ If in addition,
\begin{itemize}
\item[(4)] $H=\bigoplus_{n=0}^{\infty} H(n)$ is a coradically graded coalgebra,
\end{itemize}
then $H$ is called a \bf{coradically graded Hopf algebra}\rm. Also, the \bf{associated graded Hopf algebra}\rm\ of $H$ is defined by $\gr H=\bigoplus_{n\ge 0} H_n/H_{n-1}$ ($H_{-1}=0$) with respect to its coradical filtration. 
\end{definition}
There are some basic properties of finite-dimensional Hopf algebras, which we use frequently. 
\begin{proposition}\label{BPH}
Let $H$ be a finite-dimensional Hopf algebra.
\begin{itemize}
\item[(1)] $H$ is local if and only if $H^*$ is connected.
\item[(2)] If $H$ is local, then any quotient or Hopf subalgebra of $H$ is local.
\end{itemize}
Furthermore assume that $H$ is connected. Denote by $u\left(\Prim(H)\right)$ the restricted universal enveloping algebra of $\Prim(H)$.
\begin{itemize}
\item[(3)] Any quotient or Hopf subalgebra of $H$ is connected.
\item[(4)] $\dim \Prim(H)=\dim J/J^2$, where $J$ is the Jacobson radical of $H^*$.
\item[(5)] $H$ is primitively generated if and only if $H\cong u\left(\Prim(H)\right)$.
\item[(6)] $\dim u\left(\Prim(H)\right)=p^{\dim \Prim(H)}$.
\item[(7)] $\dim H=p^n$ for some integer $n$. 
\end{itemize}
\end{proposition}
\begin{proof}
$(1)$ and $(4)$ are derived from \cite[Prop. 5.2.9]{MO93}. 

For $(3)$ assume $H$ is connected, $H/I$ is connected by \cite[Cor. 5.3.5]{MO93}, where $I$ is any Hopf ideal of $H$. And for any Hopf subalgebra $K$ of $H$, by \cite[Lemma 5.2.12]{MO93}, $K_0=K\bigcap H_0$. Since $H_0$ is one-dimensional, so is $K_0$. Thus $K$ is connected. 

$(2)$ is the dual version of $(3)$ by $(1)$. 

$(5)$ is a standard result from \cite[Prop. 13.2.3]{S} and $(6)$ comes from \cite[P. 23]{MO93}. 

$(7)$ is true because the associated graded ring $\gr_J(H^*)$ with respect to its $J$-adic filtration is connected and primitively generated. Hence $\dim H=\dim H^*=\dim \gr_J(H^*)=p^n$, where $n=\dim \Prim(\gr_J(H^*))$ by $(6)$. 
\end{proof}
\begin{definition}\label{BDCHA}
Consider an inclusion of finite-dimensional connected Hopf algebras $K\subseteq H$.
\begin{itemize}
\item[(1)]  If $\dim K=p^m$ and $\dim H=p^n$, then the \bf{$p$-index}\rm\ of $K$ in $H$ is defined to be $n-m$.
\item[(2)]  The \bf{first order}\rm\ of the inclusion is defined to be the minimal integer $n$ such that $K_n\subsetneq H_n$. And we say it is infinity if $K=H$.  
\item[(3)]  The inclusion is said to be \bf{level-one}\rm\ if $H$ is generated by $H_n$ as an algebra, where $n$ is the first order of the inclustion. 
\item[(4)]  The inclusion is said to be \bf{normal}\rm\ if $K$ is a normal Hopf subalgebra of $H$.
\end{itemize}
\end{definition}
\begin{remark}\label{BHC}
By \cite[Lemma 5.2.12]{MO93}, if $D$ is a subcoalgebra of $C$, we have $D_n=D\bigcap C_n\subseteq C_n$. Also the coradical filtration is exhaustive for any coalgebra by \cite[Thm. 5.2.2]{MO93}. As a result of \cite[Lemma 5.2.10]{MO93}, a connected bialgebra is automatically a connected Hopf algebra. Furthermore, it is well known that any sub-bialgebra of a connected Hopf algebra is a Hopf subalgebra. Let $H$ be a connected Hopf algebra. Then the algebra generated by each term of the coradical filtration $H_n$ is a connected Hopf subalgebra of $H$. Because each term of the coradical filtration $H_n$ is a subcoalgebra and the algebra generated by it is certainly a sub-bialgebra.
\end{remark}
Throughout the whole paper we will use the following convention: 
\begin{convention}
Define the expression $\omega(x)=\sum_{i=1}^{p-1}\frac{(p-1)!}{i!(p-i)!}\ x^i\otimes x^{p-i}$, where $\frac{(p-1)!}{i!(p-i)!}\in \field$ for each $1\le i\le p-1$.
\end{convention}

\section{Associated graded Hopf algebras for finite-dimensional connected Hopf algebras}
\begin{theorem}\label{generalgrc}
Let $H=\bigoplus_{n=0}^{\infty} H(n)$ be a finite-dimensional connected coradically graded Hopf algebra. Then $H$ is isomorphic to $\field\left[x_1,x_2,\cdots,x_d\right]/\left(x_1^p,x_2^p,\cdots,x_d^p\right)$ for some $d\ge0$ as algebras.
\end{theorem}
\begin{proof}
Denote by $K=\bigoplus_{n=0}^{\infty} H(n)^*$ the graded dual of $H$. It is a graded Hopf algebra and connected for $K_0\subseteq K(0)=H(0)^*=\field$ by \cite[Lemma 5.3.4]{MO93}. Moreover since $H$ is coradically graded, by \cite[Lemma 5.5]{andruskiewitsch2000finite}, $K$ is generated in degree one and hence cocommutative. Therefore by duality $H$ is commutative and local. Then according to \cite[Thm. 14.4]{GTM66}, $H$ is isomorphic to $\field[x_1,x_2,\cdots,x_d]/(x_1^{p^{n_1}},x_2^{p^{n_2}},\cdots,x_d^{p^{n_d}})$ for some $d\ge 0$ as an algebra. Thus it suffices to prove inductively that for any homogeneous element $x\in H(n)$,  we have $x^p=0$ for all $n\ge 1$. Since $H$ is coradically graded, $\Prim(H)=H(1)$. Then for any $x\in H(1)$, we have $x^p\in (H(1))^p\bigcap H(1)\subseteq H(p)\bigcap H(1)=0$. Assume the assertion holds for $n\le m-1$. Let $x\in H(m)$. By the definition of graded Hopf algebras we have:
\begin{align*}
\Delta(x)=x\otimes 1+1\otimes x+\sum_{i=1}^{m-1}y_i\otimes z_{m-i},
\end{align*}
where $y_i,z_i\in H(i)$ for all $1\le i\le m-1$. Therefore $\Delta(x^p)=x^p\otimes 1+1\otimes x^p+\sum_{i=1}^{m-1}y_i^p\otimes z_{m-i}^p=x^p\otimes 1+1\otimes x^p$ by induction. Thus $x^p\in (H(m))^p\bigcap H(1)\subseteq H(pm)\bigcap H(1)=0$.
\end{proof}
\begin{corollary}\label{connectedgr}
The associated graded Hopf algebra of a finite-dimensional connected Hopf algebra is isomorphic to $\field\left[x_1,x_2,\cdots,x_d\right]/\left(x_1^p,x_2^p,\cdots,x_d^p\right)$ for some $d\ge0$ as algebras.
\end{corollary}
\begin{proof}
The associated graded space $\gr H=\bigoplus_{n\ge 0} H_n/H_{n-1}$ is a graded Hopf algebra by \cite[P. 62]{MO93}. Also mentioned in \cite[Def. 1.13]{Andruskiewitsch02pointedhopf}, $\gr H$ is coradically graded. Therefore $\gr H$ is a coradically graded Hopf algebra, which is clearly connected because $H$ is connected. Hence $\gr H$ satisfies all the conditions in Theorem \ref{generalgrc} and the result follows.
\end{proof}
As a consequence of the commutativity of the associated graded Hopf algebra for any finite-dimensional connected Hopf algebra we conclude that:
\begin{corollary}\label{productcoradical}
Let $H$ be a finite-dimensional connected Hopf algebra. Then $[H_n,H_m]\subseteq H_{n+m-1}$ for all integers $n,m$.
\end{corollary}

\section{Finite-dimensional connected Hopf algebras with Hopf subalgebras}
In this section, we always assume $K\subseteq H$ is an inclusion of finite-dimensional connected Hopf algebras.

\begin{lemma}\label{Contraddim}
Suppose the inclusion $K\subseteq H$ has first order $n$. Then the $p$-index of $K$ in $H$ is greater or equal to $\dim (H_n/K_n)$.
\end{lemma}
\begin{proof}
By Remark \ref{BHC}, the inclusion $K\hookrightarrow H$ induces an injection $K_i/K_{i-1}\hookrightarrow H_i/H_{i-1}$ for all $i\ge 1$. Thus $\gr K=\bigoplus_{i\ge 0} K(i)\hookrightarrow \gr H=\bigoplus_{i\ge 0} H(i)$ and $K(i)=H(i)$ for all $0\le i\le n-1$ since $n$ is the first order of the inclusion. Moreover by \cite[Def. 1.13]{Andruskiewitsch02pointedhopf}, $\left(\gr H\right)_m=\bigoplus_{0\le i\le m} H(m)$ for all $m\ge 0$ and the same is true for $\gr K$. Therefore it is enough to prove the result in the associated graded Hopf algebras inclusion $\gr K\subseteq \gr H$. 

For simplicity, we write $K$ for $\gr K$, $H$ for $\gr H$ and use $\degree(H/K)$ to denote the $p$-index of $K$ in $H$. We will prove the result by induction on $\dim (H_n/K_n)$. When $\dim (H_n/K_n)=1$, it is trivial. Now suppose that $\dim(H_n/K_n)>1$ and choose any $x\in H(n)\setminus K(n)$. Because $H$ is a graded coalgebra,
\begin{align*}
\Delta(x)=x\otimes 1+1\otimes x+\sum_{i=1}^{n-1}y_i\otimes z_{n-i},
\end{align*}
where $y_i,z_i\in H(i)=K(i)$ for all $1\le i\le n-1$. Hence $K$ and $x$ generate a Hopf subalgebra of $H$ by Remark \ref{BHC}, which we denote as $L$. Now according to Theorem \ref{generalgrc}, we have $x^p=0$. Thus $K\subseteq L$ has $p$-index one and first order $n$. Because $H$ is a graded algebra, it is clear that $L_n$ is spanned by $K_n$ and $x$. Hence $\dim (L_n/K_n)=1$ and $\dim(H_n/L_n)=\dim (H_n/K_n)-1$. Therefore by induction we have 
\begin{align*}
\dim (H_n/K_n)&=\dim (H_n/L_n)+\dim (L_n/K_n)=\dim (H_n/L_n)+1\\
&\le \degree(H/L)+1=\degree(H/L)+\degree(L/K)=\degree (H/K).
\end{align*}
\end{proof}
\begin{lemma}\label{normality}
Let $K\subseteq H$ be a level-one inclusion with first order $n$. Then $K$ is normal in $H$ if and only if $[K,H_n]\subseteq K$.
\end{lemma}
\begin{proof}
First suppose that $K$ is normal in $H$. By \cite[Lemma 5.3.2]{MO93} for any $x\in H_n$, $\Delta(x)-x\otimes 1-1\otimes x\in H_{n-1}\otimes H_{n-1}=K_{n-1}\otimes K_{n-1}\subseteq K\otimes K$. Thus we can write $\Delta(x)=x\otimes 1+1\otimes x+\sum a_i\otimes b_i$ where $a_i,b_i\in K$. Apply the antipode $S$ to get
\begin{align*}
S(x)=\e(x)-x-\sum a_iS(b_i).
\end{align*}
By the definition of normal Hopf subalgebras \cite[Def. 3.4.1]{MO93}, for any $y\in K$
\begin{align*}
\sum x_1yS(x_2)=xy+yS(x)+\sum a_iyS(b_i)=u\in K.
\end{align*}
Therefore 
\begin{align*}
[y,x]=yx-xy=y\left(\e(x)-\sum a_iS(b_i)\right)+\sum a_iyS(b_i)-u\subseteq K,
\end{align*}
which shows that $[K,H_n]\subseteq K$. Conversely suppose that $[K,H_n]\subseteq K$. Then it is clear that $K^+H_n\subseteq H_nK^++K^+\subseteq HK^+$ since $[K^+,H_n]\subseteq K^+$. We claim that $K^+(H_n)^i\subseteq HK^+$ for all $i\ge 0$ by induction. Suppose the inclusion holds for $i$ and then for $i+1$:
\begin{align*}
K^+\left(H_n\right)^{i+1}=K^+\left(H_n\right)^iH_n\subseteq \left(HK^+\right)H_n\subseteq H\left(HK^+\right)\subseteq HK^+.
\end{align*}
Therefore $K^+H=\bigcup K^+(H_n)^i\subseteq HK^+$ and by symmetry $K^+H=HK^+$. According to \cite[Cor. 3.4.4]{MO93}, $K$ is normal. 
\end{proof}

\begin{lemma}\label{normalcom}
If $x\in H$ satisfies $[K,x]\subseteq K$ and $\Delta(x)-x\otimes 1-1\otimes x \in K\otimes K$, then $\Delta\left(x^{p^n}\right)-x^{p^n}\otimes 1-1\otimes x^{p^n}\in K\otimes K$ for all $n\ge 0$.
\end{lemma}
\begin{proof}
First, we prove $\Delta \left(x^{p}\right)-x^{p}\otimes 1-1\otimes x^{p}\in K\otimes K$. Denote $\Delta(x)=x\otimes 1+1\otimes x+u$, where $u\in K\otimes K$. By Lemma \ref{palgebra}, we have:
\begin{align*}
\Delta\left(x^{p}\right)=\left(x\otimes 1+1\otimes x+u\right)^p=x^p\otimes 1+1\otimes x^{p}+u^p+\sum_{i=1}^{p-1}S_i
\end{align*}
where $iS_i$ is the coefficient of $\lambda^{i-1}$ in $u\left(\ad\left(\lambda u+x\otimes 1+1\otimes x\right)\right)^{p-1}$. Hence it suffices to show inductively that 
\begin{align*}
u\left(\ad\left(\lambda u+x\otimes 1+1\otimes x\right)\right)^n\in \left(K\otimes K\right)[\lambda]
\end{align*}
for all $n\ge 0$. Notice that when $n=0$, it is just the assumption. Suppose it's true for $n-1$ then for $n$
\begin{align*}
u\left(\ad\left(\lambda u+x\otimes 1+1\otimes x\right)\right)^n&\in\left[\left(K\otimes K\right)[\lambda],\lambda u+x\otimes 1+1\otimes x \right]\\
&\subseteq\left\{\left[K\otimes K, u\right]+[K,x]\otimes K+K\otimes [K,x]\right\}[\lambda]\\
&\subseteq \left(K\otimes K\right)[\lambda].
\end{align*}
Now replace $x$ with $x^{p^{n-1}}$ and we have $[K,x^{p^{n-1}}]=K\left(\ad(x)\right)^{p^{n-1}}\subseteq K$ by Lemma \ref{palgebra}. Then the other cases can be proved in the similar way.
\end{proof}

\begin{lemma}\label{subHopfLn}
If $x\in H$ satisfies $\Delta(x)-x\otimes 1-1\otimes x\in K\otimes K$ and $[K,x]\subseteq \sum_{0\le i\le 1}Kx^i$. For each $n\ge 0$, set $L_n=\sum_{i\le n} K x^i$. Then we have the following 
\begin{itemize}
\item[(1)] $[K,x^n]\subseteq L_n$ and $L_n$ is a $K$-bimodule via the multiplication in $H$.
\item[(2)] $\Delta(x^n)-x^n\otimes 1-1\otimes x^n\in L_{n-1}\otimes L_{n-1}$.
\item[(3)] $L_n$ is a subcoalgebra of $H$.
\item[(4)] If $H$ is generated by $K$ and $x$ as an algebra, then $H=\bigcup_{n\ge 0} L_n$.
\end{itemize}
 \end{lemma}

\begin{proof}
$(1)$ Since $xL_n\subseteq L_{n+1}$, we have $x^nL_1\subseteq L_{n+1}$ for all $n\ge 0$. By assumption, it holds that $[K,x]\subseteq L_1$. Suppose $[K,x^{n-1}]\subseteq L_{n-1}$. For any $a\in K$, it follows that 
\begin{eqnarray*}
x^na\in x^{n-1}\left(ax+L_1\right)\subseteq \left(ax^{n-1}+L_{n-1}\right)x+x^{n-1}L_1\subseteq ax^n+L_n.
\end{eqnarray*}
Hence $[K,x^n]\subseteq L_n$ for each $n\ge 0$. Moreover, we have $L_nK\subseteq L_n$ for each $n\ge 0$, the left $K$-module $L_n$ now becomes $K$-bimodule.

$(2)$ Denote $\Delta(x)=x\otimes 1+1\otimes x+u$, where $u\in K\otimes K$. We still prove by induction. When $n=1$, it is just the assumption. Suppose it's true for $n-1$. Write $\Delta(x^{n-1})=x^{n-1}\otimes 1+1\otimes x^{n-1}+\sum a_i\otimes b_i$, where $a_i,b_i\in L_{n-2}$. Therefore
\begin{align*}
&\Delta(x^n)-x^n\otimes 1-1\otimes x^n\\
&= \left(x\otimes 1+1\otimes x+u\right)\left(x^{n-1}\otimes 1+1\otimes x^{n-1}+\sum a_i\otimes b_i\right)-x^n\otimes 1-1\otimes x^n\\\
&\in x\otimes x^{n-1}+x^{n-1}\otimes x+xL_{n-2}\otimes L_{n-2}+L_{n-2}\otimes xL_{n-2}+L_{n-2}\otimes L_{n-2}\\
&\subseteq L_{n-1}\otimes L_{n-1}.
\end{align*}

$(3)$ Now because of $(1)$ and $(2)$, it is enough to check that $L_n$ is a coalgebra by induction. 

$(4)$ Furthermore if $H$ is generated by $K$ and $x$ as an algebra, it is easy to see $H=\bigcup_{n\ge 0} L_n$.
\end{proof}

\begin{theorem}\label{FREENESS}
Let $H$ be a finite-dimensional connected Hopf algebra with Hopf subalgebra $K$. Suppose the $p$-index of $K$ in $H$ is $d$ and $H$ is generated by $K$ and some $x\in H$ as an algebra. Also assume that $\Delta(x)=x\otimes 1+1\otimes x+u$, where $u\in K\otimes K$ and $[K,x]\subseteq \sum_{0\le i\le 1}Kx^i$. Then $H$ is a free left $K$-module such that $H=\bigoplus_{i=0}^{p^d-1} Kx^i$. Furthermore if $K$ is normal in $H$, then $x$ satisfies a polynomial equation as follows:
\begin{align*}
x^{p^d}+\sum_{i=0}^{d-1}a_ix^{p^i}+b=0
\end{align*} 
for some $a_i\in \field$ and $b\in K$. 
\end{theorem}
\begin{proof}
Denote $L_n=\sum_{0\le i\le n} Kx^i$ for all $n\ge 0$. By the Lemma \ref{subHopfLn}(3), $L_n$ is a subcoalgebra. Also $H$  is a left $K$-module with generators $\{x^i|i\ge 0\}$ for $H=\sum Kx^i$. Because $H$ is finite-dimensional, there exist some nontrivial relations between the generators such as
\begin{align*}
d_mx^m+d_{m-1}x^{m-1}+\cdots+d_1x+d_0=0,
\end{align*}
where $d_i\in K$ and $d_m\neq 0$, among which we choose the lowest degree in terms of $x$, say degree $m$. Furthermore denote $D=K$, $L=L_{m-1}$, $F=x^m$ and $V=\{a\in D|aF\in L\}$. As a result of Lemma \ref{subHopfLn}(2), we know $\Delta(F)-x^m\otimes 1-1\otimes x^m\in L\otimes L$. Then $D,L,F$ satisfy all the conditions listed in \cite[Lemma 1.1]{wang2011lower}. Hence $V=D$ for $0\neq d_m\in V$. Thus $x^{m}\in \bigoplus_{i<m} Kx^i$ and consequently $H$ is a free left $K$-module with the free basis $\{x^i|0\le i\le m-1\}$. Since $\dim H=m\dim K$, it is easy to see $m=p^d$ by definition.

Now assume that $K$ is normal. Follow the proof in Lemma \ref{normality}, we can show that $[K,x]\subseteq K$. From pervious discussion there exists a general equation for $x$:
\begin{align}\label{SRE}
x^{p^d}+\sum_{i=0}^ {p^d-1}a_ix^{i}=0,
\end{align}
where all $a_i\in K$. According to Lemma \ref{normalcom}, we can write $\Delta\left(x^{p^n}\right)=x^{p^n}\otimes 1+1\otimes x^{p^n}+u_n$, where $u_n\in K\otimes K$ for all $n\ge 0$. Now apply the comultiplication $\Delta$ to the above identity \eqref{SRE} to get
\begin{equation*}
x^{p^d}\otimes 1+1\otimes x^{p^d}+u_d+\sum_{i=0}^{p^d-1}\Delta(a_i)(x\otimes 1+1\otimes x+u)^{i}=0.
\end{equation*}
Replacing $x^{p^d}$ with $\left(-\sum_{i=0}^{p^d-1} a_{i}x^{i}\right)$, the following equation is straightforward: 
\begin{gather}\label{EQ1}
\left(-\sum_{i=0}^{p^d-1} a_{i}x^{i}\right)\otimes 1+1\otimes \left(-\sum_{i=0}^{p^d-1} a_{i}x^{i}\right)\\
+\sum_{i=0}^{d-1}\Delta\left(a_{p^i}\right)\left(x^{p^i}\otimes 1+1\otimes x^{p^i}+u_i\right)+\sum_{i\in S}\Delta\left(a_i\right)\left(x\otimes 1+1\otimes x+u\right)^{i}+\Delta(a_0)+u_d=0\nonumber
\end{gather}
with the $p$-index set $S=\{1,2,\cdots,p^d\}\setminus \{1,p,p^2,\cdots,p^d\}$. 

We first prove that $a_i=0$ for all $i\in S$ by contradiction. If not, suppose $n\in S$ is the largest integer such that $a_n\neq 0$. The free $K$-module structure for $H$ implies that the $K\otimes K$-module $H\otimes H$ has a free basis $\left\{x^{i}\otimes x^{j}|0\le i,j<p^d\right\}$. Thus the term $Kx^{n-i}\otimes Kx^i$ would only come from $\Delta\left(a_{n}\right)\left(x\otimes 1+1\otimes x+u\right)^{n}$ for all $1\le i\le n-1$. Moreover it exactly comes from $\Delta\left(a_{n}\right)\left(x\otimes 1+1\otimes x\right)^{n}$ by the choice of $n$. Therefore ${n\choose i}\Delta\left(a_{n}\right)\left(x^{n-i}\otimes x^i\right)=0$ for all $1\le i\le n-1$. Suppose $n=p^\alpha m$ where $m>1$ and $m\not\equiv 0 \pmod p$. Choose $i=p^\alpha$. Hence by \cite[Lemma 5.1]{isaacs1994algebra}, ${n \choose p^\alpha}\equiv m \pmod p$. Then $\Delta(a_{n})=0$, which implies that $a_n=0$, a contradiction. Therefore from equation \eqref{EQ1}, we deduce that $\Delta(a_{p^i})(x^{p^i}\otimes 1)=a_{p^{i}}x^{p^i}\otimes 1$ for all $0\le i\le d-1$. Thus $\Delta(a_{p^i})=a_{p^i}\otimes 1$. Then since $H$ is counital, all of $a_{p^i}$ are coefficients in the base field $\field$.
\end{proof}

\section{Finite-dimensional cocommutative connected Hopf algebras}
Notice that the following lemma holds over any arbitrary base field. In the remaining of this section, we still assume $\field$ to be algebraically closed of characteristic $p>0$.
\begin{lemma}\label{commNAI}
Let $H$ be a finite-dimensional Hopf algebra with normal Hopf subalgebras $K\subseteq L\subseteq H$. Then there exists a natural isomorphism:
\begin{equation*}
\left(H/K^+H\right)^*\Big/\left(H/L^+H\right)^{*+}\left(H/K^+H\right)^*\cong \left(L/K^+L\right)^*.
\end{equation*}
\end{lemma}
\begin{proof}
By \cite[Thm. 2.1.3]{MO93}, $L$ is Frobenius. Hence the injective left $L$-module map $L\hookrightarrow H$ splits since $L$ is self-injective. Therefore we can write $H=L\bigoplus M$ as a direct sum of two left $L$-modules. Because $K\subseteq L$,  we have $L\bigcap K^+H=L\bigcap K^+\left(L\bigoplus M\right)=L\bigcap \left(K^+L\bigoplus K^+M\right)=K^+L$. Then the inclusion map $L\hookrightarrow H$ induces an injective Hopf algebra map $L/K^+L\hookrightarrow H/K^+H$, since $K^+L$ and $K^+H$ are Hopf ideals of $L$ and $H$ by \cite[Lemma 3.4.2]{MO93}.

It is clear that the composition map $L/K^+L\hookrightarrow H/K^+L\twoheadrightarrow H/L^+H$ factors through $\field$ by the counit. Thus the dualized map restricted on $(H/L^+H)^{*+}=(H/L^+H)^*\bigcap \Ker\ u^*\to (L/K^+L)^*$ is the zero map, where $u$ is the unit map in $H$. 
 
Therefore the natural surjective map $(H/K^+H)^*\twoheadrightarrow (L/K^+L)^*$, which is induced by the inclusion $L/K^+L\hookrightarrow H/K^+H$, factors through $\left(H/K^+H\right)^*\Big/\left(H/L^+H\right)^{*+}\left(H/K^+H\right)^*$. In order to show that it is an isomorphism, it is enough to prove that both sides have the same dimension. By \cite[Theorem 3.3.1]{MO93}, we have
\begin{align*}
\dim \left(H/K^+H\right)^*\Big/\left(H/L^+H\right)^{*+}\left(H/K^+H\right)^*&=\dim \left(H/K^+H\right)^*\Big/\dim(H/L^+H)^*\\
&=(\dim H/\dim K)\Big/(\dim H/\dim L)\\
&=\dim L/\dim K\\
&=\dim(L/K^+L)^*.
\end{align*}
\end{proof}

Let $H$ be any Hopf algebra over $\field$, and $\field\subseteq E$ be a field extension. In the proof of \cite[Cor. 2.2.2]{MO93}, we know that $H\otimes E$ is also a Hopf $E$-algebra, via
\begin{align*}
\Delta(h\otimes \alpha)&:=\Delta (h)\otimes \alpha\in H\otimes H\otimes E\cong (H\otimes E)\otimes_E(H\otimes E)\\
\e(h\otimes \alpha)&:=\e(h)\alpha\in E\\
S(h\otimes \alpha)&:=S(h)\otimes \alpha
\end{align*}
for all $h\in H,\alpha\in E$. Now consider any automorphism $\sigma$ of $\field$. By taking $E=\field$ and $\sigma$ to be the embedding in the discussion above, $H\otimes_\sigma\field$ is also a Hopf $\field$-algebra, which we will denote by $H_\sigma$. Note that in $H_\sigma$, we have $h\alpha\otimes 1=h\otimes \sigma(\alpha)$ for all $h\in H,\alpha\in \field$. Let $id_\sigma$ be the map $id\otimes 1$ from $H$ to $H_{\sigma}$. The following hold for all $h, l\in H$ and $\alpha\in \field$
\begin{gather*}
id_\sigma(hl)=id_\sigma(h)id_\sigma(l),\ \Delta id_\sigma(h)=(id_\sigma\otimes id_\sigma)\Delta h,\ S(id_\sigma(h))=id_\sigma(S(h))\\
\e id_\sigma(h)=\sigma\left(\e(h)\right),\ id_\sigma(h\alpha)=id_\sigma(h)\sigma (\alpha).
\end{gather*}
Generally, let $A$ be another Hopf algebra over $\field$, and $\phi$ be a map from $A$ to $H$. We say that $\phi: A\mapsto H$ is a \textbf{$\sigma$-linear Hopf algebra map} if the composition $id_\sigma\circ\phi: A\mapsto H_\sigma$ is a $\field$-linear Hopf algebra map. Suppose $H,A$ are both finite-dimensional. Note that $(H_\sigma)^*\cong (H^*)_\sigma$ since $\Hom_E(H\otimes E,E)\cong \Hom_\field(H,\field)\otimes E$ for any field extension $\field\subseteq E$. Let $f$ be a $\sigma$-linear Hopf algebra map from $A$ to $H$. It is clear that the dual of $f$ is a $\sigma^{-1}$-linear Hopf algebra map from $H^*$ to $A^*$. Also quotients of $\sigma$-linear Hopf algebra maps are still $\sigma$-linear.

\begin{proposition}\label{cocommfiltration}
Let $H$ be a finite-dimensional cocommutative connected Hopf algebra. Then $H$ has an increasing sequence of normal Hopf subalgebras: $\field=N_0\subset N_1\subset \cdots\subset N_n=H$ satisfying the following properties:
\begin{itemize} 
\item[(1)] Denote by $J$ the Jacobson radical of $H^*$. Then the length $n$ is the minimal integer such that $x^{p^n}=0$ for all $x\in J$.
\item[(2)] $N_1$ is the Hopf subalgebra of $H$ generated by all primitive elements.
\item[(3)] There are $\sigma$-linear injective Hopf algebra maps: 
\[\xymatrix{
N_{m}/N_{m-1}^+N_m\ar@{^(->}[r]& N_{m-1}/N_{m-2}^+N_{m-1}
}\]
for all $2\le m\le n$, where $\sigma$ is the Frobenius map of $\field$.
\item[(4)] $0=\dim \Prim\left(H/N_n^{+}H\right)\le \dim \Prim\left(H/N_{n-1}^{+}H\right)\le\cdots\le \dim \Prim\left(H/N_0^{+}H\right)=\dim \Prim(H)$.
\end{itemize} 
\end{proposition}
\begin{proof}
$(1)$ By duality, $H^*$ is a finite-dimensional commutative local Hopf algebra. Therefore by \cite[Thm. 14.4]{GTM66} we can write:
\begin{align*}
H^*=\field\left[x_1,x_2,\cdots,x_d\right]\Big /\left(x_1^{p^{n_1}},x_2^{p^{n_2}},\cdots,x_d^{p^{n_d}}\right)
\end{align*} 
for some $d\ge 0$, in which we can define a decreasing sequence of normal Hopf ideals \cite[Def. 3.4.5]{MO93}
\begin{align*}
\left(J_m=(x_1^{p^{m}},x_2^{p^{m}},\cdots,x_d^{p^{m}})\right)_{m\ge 0}.
\end{align*}
By \cite[P. 36]{MO93}, in the dual vector space $H$ we have an increasing sequence of normal Hopf subalgebras: $\field=N_0\subset N_1\subset \cdots\subset N_m\subseteq\cdots\subseteq H$, where $N_m=\left(H^*/J_m\right)^*$ for all $m\ge 0$.  For the length of this sequence, notice that $N_m=H\Leftrightarrow J_m=0\Leftrightarrow x_i^{p^m}=0$ for all $1\le i\le d\Leftrightarrow x^{p^m}=0$ for all $x\in J_0=J$. 

$(2)$ Denote by $L$ the Hopf subalgebra of $H$ generated by $\Prim(H)$. By \cite[Prop. 5.2.9]{MO93}, $\field\bigoplus \Prim(H)=\{h\in H|\langle J^2,h\rangle=0\}$. Hence under the natural identification, $\Prim(H)\subset(H^*/J^2)^*\subseteq (H^*/J_1)^*=N_1$. Because $L$ is generated by $\Prim(H)$ as an algebra, we have $L\subseteq N_1$. Moreover we know $\dim L=p^{\dim \Prim(H)}=p^{\dim J/J^2}=p^d$ by Proposition \ref{BPH}(4). On the other side, $\dim N_1=\dim H^*/J_1=p^d$, which implies that $L=N_1$. 

$(3)$ Define a decreasing sequence of normal Hopf subalgebras of $H^*$ by
\begin{align*}
A_{m}=\{h^{p^m}|h\in H^*\}=\field\left[x_1^{p^{m}},x_2^{p^{m}},\cdots,x_d^{p^{m}}\right].
\end{align*}
Notice that $A_m^+H^*=J_m$ for all $m\ge 0$. Moreover, by Lemma \ref{commNAI}, we have
\begin{align}\label{NIF}
\left(A_m/A_{m+1}^+A_m\right)^*&\cong \left(H^*/A_{m+1}^+H^*\right)^*\Big/\left(H^*/A_m^+H\right)^{*+}\left(H^*/A_{m+1}^+H^*\right)^*\\
&=N_{m+1}\Big/N_m{^+}N_{m+1}.\notag
\end{align} 
Let $\sigma$ be the Frobenius map of $\field$ (i.e., the $p$-th power map). For any $2\le m\le n$, we can take $(A_{m-2})_{\sigma^{-1}}=A_{m-2}\otimes_{\sigma^{-1}} \field$ such that $ak\otimes 1=a\otimes \sigma^{-1}(k)$ for any $a\in A_{m-2}$ and $k\in \field$. Hence it is easy to see that there exists a series of $\sigma^{-1}$-linear surjective $p$-th power Hopf algebra maps $\phi_{m-2}: A_{m-2} \twoheadrightarrow A_{m-1}$ such that $\phi_{m-2}(x)=x^p$ for all $x\in A_{m-2}$. Therefore $\phi_{m-2}$ induces a series of $\sigma^{-1}$-linear surjective maps on their quotients $A_{m-2}/A_{m-1}^+A_{m-2}\twoheadrightarrow A_{m-1}/A_{m}^+A_{m-1}$. By daulizing all the maps and the above natural isomorphism \eqref{NIF}, we have a series of $\sigma$-linear injective Hopf algebra maps:
\[\xymatrix{
N_{m}/N_{m-1}^+N_m\ar@{^(->}[r]& N_{m-1}/N_{m-2}^+N_{m-1}
}\]
for all $2\le m\le n$.

$(4)$ In Lemma \ref{commNAI}, let $K=\field$ and $L=A_m$. Then we have the special isomorphism:
\begin{align*}
A_m^*\cong H\Big/N_m^+H.
\end{align*} 
Therefore, by Proposition \ref{BPH}(4), 
\begin{align*}
\dim \Prim(H/N_m^{+}H)=\dim J(A_m)/J(A_m)^2=\#\left\{\{x_1^{p^{m}},x_2^{p^{m}},\cdots,x_d^{p^{m}}\}\setminus \{0\}\right\},
\end{align*}
which is the number of generators among $\{x_1,x_2,\cdots,x_d\}$, whose $p^{m}$-th power does not vanish. Thus the inequalities follow.
\end{proof}

\begin{corollary}\label{FCLH}
Let $H$ be a finite-dimensional connected Hopf algebra with $\dim\Prim(H)=1$. Then $H$ has an increasing sequence of normal Hopf subalgebras:
\begin{align*}
\field=N_0\subset N_1\subset N_2\subset \cdots \subset N_n=H,
\end{align*} 
where $N_1$ is generated by $\Prim(H)$ and each $N_i$ has $p$-index one in $N_{i+1}$.
\end{corollary}
\begin{proof}
Denote by $H^*$ the dual Hopf algebra of $H$. By duality, $H^*$ is local. Set $J=J(H^*)$, the Jacobson radical of $H^*$. Since $\dim \Prim(H)=1$, by Proposition \ref{BPH}(4), $\dim J/J^2=1$. Suppose that $\dim H=p^n$ by Proposition \ref{BPH}(7). It is clear that $H^*\cong \field\left[x\right]/(x^{p^n})$ as algebras and $J=(x)$. Hence $H$ is cocommutative and it has an increasing sequence of normal Hopf subalgebras $\field=N_0\subset N_1\subset \cdots\subset N_n=H$ such that $N_1$ is generated by $\Prim(H)$ and $\dim N_m=p^m$ for all $0\le m\le n$ by Proposition \ref{cocommfiltration}.
\end{proof}

\begin{theorem}\label{NPLA}
Let $H$ be finite-dimensional cocommutative connected Hopf algebra. Denote by $K$ the Hopf subalgebra generated by $\Prim(H)$. Then the following are equivalent:
\begin{itemize}
\item[(1)] $H$ is local.
\item[(2)] $K$ is local.
\item[(3)] All the primitive elements of $H$ are nilpotent.
\end{itemize}
\end{theorem}
\begin{proof}
$(1)\Rightarrow (2)$ is from Proposition \ref{BPH}(2) and $(2)\Rightarrow (3)$ is clear since $K$ contains $\Prim(H)$ and its augmentation ideal is nilpotent. 

In order to show that $(3)\Rightarrow (2)$, denote $\mathfrak g=\Prim(H)$, which is a restricted Lie algebra. Then $(3)$ is equivalent to the statement that $\mathfrak g^{p^n}=0$ for sufficient larger $n$. Therefore $(\ad x)^{p^n}=\ad(x^{p^n})=0$ for all $x\in \mathfrak g$. By Engel's Theorem \cite[I \S 3.2]{GTM9}, $\mathfrak g$ is nilpotent. Any representation of $K\cong u(\mathfrak g)$ is a restricted representation of $\mathfrak g$. Therefore any irreducible representation of $K$ is one-dimensional with trivial action of the augmentation ideal of $K$. Hence the augmentation ideal of $K$ is nilpotent and $K$ is local. 

Finally, we need to show $(2)\Rightarrow (1)$.  Suppose $\field=N_0\subset N_1\subset \cdots N_n=H$ is the sequence of normal Hopf subalgebras stated in Proposition \ref{cocommfiltration} for $H$. By Proposition \ref{cocommfiltration}(2), we know $N_1=K$ is local. We will show inductively that each $N_m$ is local. Assume $N_m$ to be local and denote $\sigma$ as the Frobenius map of $\field$. We have the following injective Hopf algebra map according to Proposition \ref{cocommfiltration}(3) and the definition of $\sigma$-linear Hopf algebra maps:
\[\xymatrix{
N_{m+1}/N_{m}^+N_{m+1}\ar@{^(->}[r]& \left(N_{m}/N_{m-1}^+N_{m}\right)_\sigma.
}\]
Note that any finite-dimensional Hopf algebra $A$ is local if and only if its augmented ideal $A^+$ is nilpotent. Since $(A\otimes_\sigma \field)^+=(A^+)\otimes_\sigma\field$, we see that $A$ is local if and only if $A_\sigma$ is local. Hence $\left(N_{m}/N_{m-1}^+N_{m}\right)_\sigma$ is local. Moreover, by Proposition \ref{BPH}(2), $N_{m+1}/N_{m}^+N_{m+1}$ is local. Therefore there exist integers $l,d$ such that $(N_{m+1}^+)^d\subseteq N_m^+N_{m+1}$ and $(N_m^+)^l=0$. Hence $(N_{m+1}^+)^{ld}\subseteq (N_m^+)^dN_{m+1}=0$. Here we have used $N_m^+N_{m+1} = N_{m+1}N_m^+$, which follows from \cite[Cor. 3.4.4]{MO93} and the fact that $N_m$ is normal. This completes the proof.
\end{proof}

\begin{remark}\label{GSL}
Let $G$ be a connected affine algebraic group scheme over $\field$, and $G_1$ be the first Frobenius kernel of G.  By \cite[Prop. 4.3.1 Exp. XVII]{SGA3}, we know that $G$ is unipotent if and only if $\Lie\left(G\right)$ is unipotent, i.e., for any $x\in \Lie(G_1)$, there exists integer $n>0$, such that $x^{p^n}=0$. Moreover, $\Lie\left(G\right)=\Lie\left(G_1\right)$. Hence $G$ is unipotent if and only if $G_1$ is unipotent. Denote the coordinate ring $A=\field[G]$. Then $\field[G_1] = A/A^{+(p)}A$, where $A^{(p)}=\{a^p\ |\ a\in A\}$. We can state the above assertion in another way: $A$ is connected if and only if $A/A^{+(p)}A$ is connected. If $A$ is finite-dimensional, as shown in Proposition \ref{cocommfiltration}(2), $\left(A/A^{+(p)}A\right)^*$ is the Hopf subalgebra of $A^*$ generated by its primitive elements. This provides an alternative proof for Theorem \ref{NPLA} and shows that the locality criterion in Theorem \ref{NPLA} for finite-dimensional cocommutative connected Hopf algebras parallel the criteria for unipotency of finite connected group schemes over $\field$.
\end{remark}

\section{Hochschild cohomology of restricted universal enveloping algebras}
Suppose $H$ is a Hopf algebra. Denote by $\field$ the trivial $H$-bicomodule. The Hochschild cohomology $\HL^\bullet(\field,H)$ of $H$ with coefficients in $\field$ can be computed as the homology of the differential graded algebra $\Omega H$ defined as follows \cite[Lemma 1.1]{cstefan1998hochschild}: 
\begin{itemize}
\item As a graded algebra, $\Omega H$ is the tensor algebra $T(H)$,
\item The differential in $\Omega H$ is given by $d^0=0$ and for $n\ge 1$
\begin{align*}
d^n=1\otimes I_n+\sum_{i=0}^{n-1} (-1)^{i+1} I_i\otimes\Delta\otimes I_{n-i-1}+(-1)^{n+1}I_n\otimes 1.
\end{align*}
\end{itemize}
This DG algebra is usually called the \textbf{cobar construction} of $H$. See \cite[\S 19]{GTM205} for the basic properties of cobar constructions. Throughout, we will use $\HL^\bullet(\field,H)$ to denote the homology of the DG algebra $(\Omega H,d)$.

\begin{lemma}\label{DimExtCo}
Let $H$ be a finite-dimensional Hopf algebra. Thus
\begin{eqnarray*}
\HL^n\left(\field,H\right)\cong \HL^n\left(H^*,\field\right)\cong \Ext^n_{H^*}\left(\field,\field\right),
\end{eqnarray*}
for all $n\ge 0$.
\end{lemma}
\begin{proof}
We still denote by $\field$ the trivial $H$-bimodule. Then the first isomorphism comes from \cite[Prop. 1.4]{cstefan1998hochschild}. Let $M$ be a $H$-bimodule with the trivial right structure. We define the right structure of $M^{\ad}$ by $m.h=S(h)m$ using the antipode $S$ of $H$ for any $m\in M,h\in H$. Then it is easy to see $\field^{\ad}\cong \field$ as trivial right $H$-modules. Hence the second isomorphism is derived from \cite[Thm. 1.5]{cstefan1998hochschild}.
\end{proof}

Let $\mathfrak g$ be a restricted Lie algebra. We denote by $u(\mathfrak g)$ the restricted universal enveloping algebra of $\mathfrak g$. Analogue to ordinary Lie algebras, restricted $\mathfrak g$-modules are in one-to-one correspondence with $u(\mathfrak g)$-modules, i.e., a vector space $M$ is a restricted $\mathfrak g$-module if there exists an algebra map $T: \mathfrak u(\mathfrak g)\to \End_\field(M)$.

\begin{proposition}\label{Liealgebrainclusion}
Let $\mathfrak g$ be a restricted Lie algebra with basis $\{x_1,x_2,\cdots,x_n\}$. Then the image of 
\begin{align*}
\left\{\omega(x_i),\ x_j\otimes x_k\ |\ 1\le i\le n,1\le j<k\le n\right\}
\end{align*}
is a basis in $\HL^2\left(\field,u(\mathfrak g)\right)$.
\end{proposition}
\begin{proof}
Denote $K=u\left(\mathfrak g\right)$ and let $C_p^n$ be the elementary abelian $p$-group of rank $n$. It is clear that $K^*$ is isomorphic to $\field [C_p^n]$ as algebras. Then it follows from, e.g., \cite[P. 558 (4.1)]{QAST} that $\dim \HL^2(K^*, \field)=\dim \HL^2(C_p^n, \field)=n(n+1)/2$. Thus by Lemma \ref{DimExtCo}, $\dim \HL^2(\field,K)=n(n+1)/2$. First, it is direct to check that all $\omega(x_i)$ and $x_j\otimes x_k$ are cocycles in $\Omega K$. We only check for $x_j\otimes x_k$ here. Notice that $d^2=1\otimes I\otimes I-\Delta\otimes I+I\otimes \Delta-I\otimes I\otimes 1$. Thus
\begin{align*}
d^2\left(x_j\otimes x_k\right)&=1\otimes x_j\otimes x_k-\Delta(x_j)\otimes x_k+x_j\otimes \Delta(x_k)-x_j\otimes x_k\otimes 1\\
&=1\otimes x_j\otimes x_k-(x_j\otimes 1+1\otimes x_j)\otimes x_k+x_j\otimes(x_k\otimes 1+1\otimes x_k)-x_j\otimes x_k\otimes 1\\
&=0.
\end{align*}
Secondly, we need to show they are linearly independent in $\HL^2(\field,K)=\Ker\ d^2/\Img\ d^1$. We only deal with the case when $p\ge 3$. The remaining case of $p=2$ is similar. By the \PBW\ Theorem, $K$ has a basis formed by
\begin{align*}
\left\{x_1^{i_1}\ x_2^{i_2}\cdots x_n^{i_n}\ |\ 0\le i_1,i_2,\cdots,i_n\le p-1\right\}.
\end{align*}
Because the differential $d^1=1\otimes I-\Delta+I\otimes 1$ in $\Omega K$ only uses the comultiplication, without loss of generality, we can assume $\mathfrak g$ to be abelian. Suppose each variable $x_i$ of $K$ has degree one. Assign the usual total degree to any monomial in $K$. Also the total degree of a tensor product $A\otimes B$ in $K\otimes K$ is the sum of the degrees of $A$ and $B$ in $K$. Therefore $d^1$ preserves the degree from $K$ to $K\otimes K$ for any monomial. Notice that $\omega(x_i)$ has degree $p$ and $x_j\otimes x_k$ has degree two. We can treat them separately. Suppose that $\sum_i \alpha_i\omega(x_i)\in \Img d^1$. First, we consider the ideal $I=(x_2,\cdots,x_n)$ in $K$. By passing to the quotient $K/I$, we have $\alpha_1\omega(\overline{x_1})\in \Img\ \overline{d^1}$, where $\overline{d^1}: K/I\to K/I\otimes K/I$. But every monomial in $K/I$, which is generated by $x_1$, has degree less than $p$. This forces that $\alpha_1=0$. The same argument works for all the coefficients. Now suppose $\sum_{j< k} \alpha_{jk}x_j\otimes x_k\in \Img\ d^1$. Therefore there exists $\sum_{j\le k} \lambda_{jk}x_jx_k\in K$ such that
\begin{align*}
\sum_{j< k} \alpha_{jk}\ x_j\otimes x_k&=d^1\left(\sum_{j\le k} \lambda_{jk}\ x_jx_k\right)\\
&=\sum_{j\le k}\lambda_{jk} \left(1\otimes x_jx_k-\Delta(x_jx_k)+x_jx_k\otimes 1\right)\\
&=-\sum_{j\le k}\lambda_{jk}\left(x_j\otimes x_k+x_k\otimes x_j\right).
\end{align*}
By applying the \PBW\ Theorem to $K\otimes K$, we have all the coefficients equal zero. This completes the proof.
\end{proof}

\begin{lemma}\label{combinecohomologyclass}
Let $\mathfrak g$ be a restricted Lie algebra. Then the cocycle
\begin{align*}
\sum_{i=1}^{n} \alpha_i^p\ \omega\left(x_i\right)-\omega\left(\sum_{i=1}^n \alpha_i\ x_i\right)
\end{align*}
is zero in $\HL^2\left(\field,u(\mathfrak g)\right)$, where $x_i\in \mathfrak g$ and $\alpha_i\in \field$ for all $1\le i\le n$.
\end{lemma}
\begin{proof}
Denote by $K$ the restricted universal enveloping algebra of $\mathfrak g$. First, it is direct to check that $\omega(x)$ is a cocycle in $(\Omega K,d)$ for any $x\in \mathfrak g$. Hence the expression in the statement is also a cocycle in $(\Omega K,d)$. We only need to show that it lies in the coboundary $\Img\ d^1$. Without loss of generality, we can assume $\mathfrak g$ to be finite-dimensional. Because $\field$ is algebraically closed in $\mathbb F_p$, we can replace $\field$ with some finite field $\mathbb F_q$. By basic algebraic number theory, there exists some number field $L\supset \mathbb Q$, where $p$ remains prime in the ring of integers $\mathcal O_L$ such that $\mathcal O_L/(p)=\mathbb F_q$. Now by choosing representatives for $\mathbb F_q$ in $\mathcal O_L$, we can view $\mathfrak g$ as a free module over $\mathcal O_L$ with a Lie bracket $[\ ,\ ]$, representing all the relations between a chosen basis for $\mathfrak g$. Denote by $A=\mathcal U(\mathfrak g)$ the universal enveloping algebra of $\mathfrak g$ over $\mathcal O_L$, which is a Hopf algebra as usual. There is a quotient map $\pi: A\to u(\mathfrak g)$, which factors through $A/(p)$. Therefore it suffices to prove that for any $x,y\in \mathfrak g$, there exists some $\Theta\in A$ such that
\begin{align}\label{addomega}
\omega(x)+\omega(y)-\omega(x+y)=1\otimes \Theta-\Delta(\Theta)+\Theta\otimes 1.
\end{align}
The general result will follow by applying the quotient map $\pi$ to \eqref{addomega}, and the induction on the number of variables appearing in the expression. By Lemma \ref{palgebra}, in $A\otimes_{\mathcal O_L}\mathcal O_L/(p)=A\otimes_{\mathcal O_L}\mathbb F_q=A/(p)$, there exists some $z\in \mathfrak g$ such that
\begin{align*}
(x+y)^p=x^p+y^p+z.
\end{align*} 
So back in $A$, we have some $\Theta \in A$ such that
\begin{align*}
(x+y)^p=x^p+y^p+z+p\ \Theta.
\end{align*} 
Thus in $A$, we can calculate $\Delta(x+y)^p$ in two different ways:
\begin{align*}
\Delta(x+y)^p&=\left(\Delta(x+y)\right)^p\tag{I}\\
&=\left((x+y)\otimes 1+1\otimes (x+y)\right)^p\\
&=(x+y)^p\otimes 1+1\otimes (x+y)^p+p\ \omega(x+y)\\
&=(x^p+y^p+z)\otimes 1+1\otimes (x^p+y^p+z)+p\ \Theta \otimes 1+1\otimes p\ \Theta+p\ \omega(x+y).
\end{align*}
On the other hand,
\begin{align*}
\Delta(x+y)^p&=\Delta\left(x^p+y^p+z+p\ \Theta \right)\tag{II}\\
&=x^p\otimes 1+1\otimes x^p+p\ \omega (x)+y^p\otimes 1+1\otimes y^p+p\ \omega(y)+z\otimes 1+1\otimes z+p\ \Delta(\Theta)\nonumber\\
&=(x^p+y^p+z)\otimes 1+1\otimes (x^p+y^p+z)+p\ \omega(x)+p\ \omega(y)+p\ \Delta(\Theta).\nonumber
\end{align*}
Therefore we have the following identity in $A\otimes A$.
\begin{align*}
p\ \{\omega(x)+\omega(y)-\omega(x+y)\}=p\ \{1\otimes \Theta-\Delta(\Theta)+\Theta\otimes 1\}.
\end{align*}
Since $A$ is a domain, we can cancel $p$ from both sides. This completes the proof.
\end{proof}

\begin{definition}
Let $H$ be a Hopf algebra. For any $x\in H$, define the adjoint map $T_x$ on $\Omega H$ by 
\begin{align*}
T^n_x=\sum_{i=0}^{n-1} I_i\otimes \ad (x)\otimes I_{n-i-1},
\end{align*}
where $\ad(x)(H)=[x,H]$.
\end{definition}

\begin{lemma}\label{chainmap}
If $H$ is any Hopf algebra, then $T_x$ is a degree zero cochain map from $\Omega H$ to itself for all $x\in \Prim(H)$. Moreover,  $\Prim(H)=\HL^1(\field,H)$ and $\bigoplus_{n\ge 0}\HL^n\left(\field,H\right)$ is a graded restricted $\Prim(H)$-module via the adjoint map.
\end{lemma}
\begin{proof}
First, for simplicity write $T=T_x$ for some $x\in \Prim(H)$. We prove $d^nT^n=T^{n+1}d^n$ inductively for all $n\ge 0$. It is easy to check that it holds for $n=0,1$. Notice that
\begin{eqnarray*}
d^n=d^{n-1}\otimes I+(-1)^{n-1}I_{n-1}\otimes d^1,
\end{eqnarray*} 
for all $n\ge 2$. Thus
\begin{align*}
&d^n T^n\\
&=\left(d^{n-1}\otimes I+(-1)^{n-1}I_{n-1}\otimes d^1\right)\left(T^{n-1}\otimes I+I_{n-1}\otimes T^1\right)\\
&=d^{n-1}T^{n-1}\otimes I+d^{n-1}\otimes T^1+(-1)^{n-1}T^{n-1}\otimes d^1+(-1)^{n-1}I_{n-1}\otimes d^1T^1\\
&=T^{n}d^{n-1}\otimes I+d^{n-1}\otimes T^1+(-1)^{n-1}T^{n-1}\otimes d^1+(-1)^{n-1}I_{n-1}\otimes T^2d^1\\
&=T^{n}d^{n-1}\otimes I+d^{n-1}\otimes T^1+(-1)^{n-1}\left(T^{n-1}\otimes I_2+I_{n-1}\otimes T^1\otimes I\right)\left(I_{n-1}\otimes d^1\right)+(-1)^{n-1}I_{n-1}\otimes (I\otimes T^1)d^1\\
&=T^{n}d^{n-1}\otimes I+d^{n-1}\otimes T^1+(-1)^{n-1}\left(T^{n}\otimes I\right)\left(I_{n-1}\otimes d^1\right)+(-1)^{n-1}I_{n-1}\otimes \left(I\otimes T^1\right)d^1\\
&=\left(T^{n}\otimes I+I_n\otimes T^1\right)\left(d^{n-1}\otimes I+(-1)^{n-1}I_{n-1}\otimes d^1\right)\\
&=T^{n+1}d^n
\end{align*}
Therefore $T$ induces an action of $\Prim(H)$ on $\HL^n(\field,H)$ for each $n$. Moreover, we know $\Prim(H)$ is a restricted Lie algebra via the $p$-th power map in $H$. It is clear that $[T_x,T_y]=T_{[x,y]}$ and $T_x^p=T_{x^p}$ for any $x,y\in \Prim(H)$. Hence $\bigoplus_{n\ge 0}\HL^n\left(\field,H\right)$ becomes a graded restricted $\Prim(H)$-module via $T$. Finally, $\Prim(H)\cong \HL^1(\field,H)$ by definition.
\end{proof}

\begin{theorem}\label{Cohomologylemma}
Let $K\subseteq H$ be an inclusion of connected Hopf algebras with first order $n\ge 2$. Then the differential $d^1$ induces an injective restricted $\mathfrak g$-module map
\[
\xymatrix{
H_n/K_n\ar@{^(->}[r]&\HL^2(\field,K),
}\]
where $\mathfrak g=\Prim(H)$.
\end{theorem}

\begin{proof}
By Corollary \ref{productcoradical}, $H_n$ becomes a restricted $\mathfrak g$-module via the adjoint action since $[\Prim(H),H_n]\subseteq [H_1,H_n]\subseteq H_n$. We know $\mathfrak g=\Prim(H)=\Prim(K)$ for the inclusion has first order $n\ge 2$. Hence the $\mathfrak g$-action factors through $H_n/K_n$. Choose any $x\in H_n$. We know $d^1(x)=1\otimes x-\Delta(x)+x\otimes 1\in H_{n-1}\otimes H_{n-1}=K_{n-1}\otimes K_{n-1}\subseteq K\otimes K$ by \cite[Lemma 5.3.2]{MO93}. Furthermore, we can view $(\Omega K,d_K)$ as a subcomplex of $(\Omega H,d_H)$. Then $d_K^2d_H^1(x)=d_H^2d_H^1(x)=0$. Hence $d^1(x)$ is a cocycle in $\Omega K$ and $d^1$ maps $H_n$ into $\HL^2(\field,K)$. The map $d^1$ factors through $H_n/K_n$ for $d^2d^1(K_n)=0$. To show the induced map is injective, suppose $d^1(x)\in\Img\ d_K^1$. Then there exists some $y\in K$ such that $d^1(x)=d^1(y)$, which implies that $d^1(x-y)=0$. By definition, we have $x-y\in \Prim(H)=\Prim(K)$. Hence $x\in K\bigcap H_n=K_n$ by Remark \ref{BHC}. Finally, $d^1$ is compatible with the $\mathfrak g$-action on $\HL^2(\field,K)$ by Lemma \ref{chainmap}.
\end{proof}

\begin{theorem}\label{HCT}
Let $\mathfrak g$ be a restricted Lie algebra with basis $\{x_1,x_2,\cdots,x_n\}$. Suppose $u(\mathfrak g)\subsetneq H$ is an inclusion of connected Hopf algebras. Then there exists some $x\in H\setminus u(\mathfrak g)$ such that
\begin{align*}
\Delta(x)=x\otimes 1+1\otimes x+\omega\left(\sum_i\alpha_ix_i\right)+\sum_{j<k}\alpha_{jk}x_j\otimes x_k
\end{align*}
with coefficients $\alpha_i,\alpha_{jk}\in \field$. Moreover,  the first order for the inclusion can only be $1$, $2$ or $p$. 
\end{theorem}
\begin{proof}
Denote by $d$ the first order for the inclusion. By definition, $d=1$ implies that $\mathfrak g\subsetneq \Prim(H)$. Then we can find some primitive element $x\in \Prim(H)\setminus \mathfrak g\subseteq H\setminus u(\mathfrak g)$ such that $\Delta(x)=x\otimes 1+1\otimes x$. In the following, we may assume $d\ge 2$. By Theorem \ref{Cohomologylemma} and Proposition \ref{Liealgebrainclusion}, there exists $x\in H_d\setminus u(\mathfrak g)$ such that
\begin{align*}
1\otimes x-\Delta(x)+x\otimes 1=d^1(x)=-\sum_i \alpha_i^p\ \omega(x_i)-\sum_{j<k}\alpha_{jk}\ x_j\otimes x_k.\tag{I}\label{E1}
\end{align*}
By the choice of $x$, we know the coefficients are not all zero. By Lemma \ref{combinecohomologyclass}, there exists some $y\in u(\mathfrak g)$ such that
\begin{align*}
1\otimes y-\Delta(y)+y\otimes 1=d^1(y)=\sum_{i}\alpha_i^p\ \omega(x_i)-\omega\left(\sum_{i}\alpha_i\ x_i\right).&\tag{II}\label{E2}
\end{align*}
If we add \eqref{E1} to \eqref{E2}, then we have
\begin{align*}
(x+y)\otimes 1-\Delta(x+y)+1\otimes (x+y)=-\omega\left(\sum_{i}\alpha_i\ x_i\right)-\sum_{j<k}\alpha_{jk}\ x_j\otimes x_k.
\end{align*}
This implies that
\[
\Delta(x+y)=(x+y)\otimes 1+1\otimes (x+y)+\omega\left(\sum_{i}\alpha_i\ x_i\right)+\sum_{j<k}\alpha_{jk}\ x_j\otimes x_k.
\]
It is clear that $x+y\in H\setminus u(\mathfrak g)$. Finally, because the associated graded Hopf algebra $\gr H$ is coradically graded as mentioned in \cite[Def. 1.13]{Andruskiewitsch02pointedhopf}, it is easy to see that if all $\alpha_i=0$ then $d=2$. Otherwise $d=p$. Hence the first order $d$ can only be $1$, $2$ or $p$. This completes the proof.
\end{proof}

\section{Connected Hopf algebras of dimension $p^2$}
The starting point for classifying finite-dimensional connected Hopf algebras turns out to be when the dimension of the Hopf algebras is just $p$. It is obvious that such Hopf algebras are primitively generated, i.e., by some primitive element $x$. As a consequence of the characteristic of the base field, $x^p$ is still primitive. This implies that $x^p=\lambda x$ for some $\lambda\in\field$, since the dimension of  the primitive space is one. By rescaling of the variable, we can always assume the coefficient $\lambda$ to be zero or one. Thus we have the following result:
\begin{theorem}\label{D1}
All connected Hopf algebras of dimension $p$ are isomorphic to either $\field[x]/(x^p)$ or $\field[x]/(x^p-x)$, where $x$ is primitive.
\end{theorem}
\begin{corollary}
All local Hopf algebras of dimension $p$ are isomorphic to $\field[x]/(x^p)$ with comultiplication either $\Delta(x)=x\otimes 1+1\otimes x$ or $\Delta(x)=x\otimes 1+1\otimes x+x\otimes x$.
\end{corollary}
\begin{proof}
By Proposition \ref{BPH}(1), $p$-dimensional local Hopf algebras are in one-to-one correspondence with $p$-dimensional connected Hopf algebras by vector space dual. Therefore by Theorem \ref{D1}, there are two non-isomorphic classes of local Hopf algebras of dimension $p$. It is clear that $\field [x]/(x^p)$ is a local algebra of dimension $p$. Regarding the coalgebra structure, when $\Delta(x)=x\otimes 1+1\otimes x$, it is connected. When $\Delta(x)=x\otimes 1+1\otimes x+x\otimes x$, $\Delta(x+1)=(x+1)\otimes (x+1)$, which is a group-like element. Therefore it is cosemisimple. They are certainly non-isomorphic as coalgebras. 
\end{proof}
In the rest of the section, we concentrate on the classification of connected Hopf algebras of dimension $p^2$. We first consider the case when $\dim \Prim(H)=1$. By Corollary \ref{FCLH}, we have $\field\subset K\subset H$, where $K$ is generated by some $x\in \Prim(H)$. By Proposition \ref{BPH}(5), we know $K$ is isomorphic to the restricted universal enveloping algebra of the one-dimensional restricted Lie algebra spanned by $x$. Therefore by Proposition \ref{Liealgebrainclusion}, $\HL^2(\field,K)$ is one-dimensional with the basis representing by the element 
\begin{align*}
\omega(x)=\sum_{i=1}^{p-1}\ \frac{(p-1)!}{i!(p-i)!}\ x^i\otimes x^{p-i}.
\end{align*} 
Furthermore, by Theorem \ref{HCT}, there exists some $y\in H\setminus K$ such that $\Delta\left(y\right)=y\otimes 1+1\otimes y+\omega(x)$.

\begin{lemma}\label{D2P1C}
Let $H$ be a connected Hopf algebra of dimension $p^2$ with $\dim\Prim(H)=1$. Then $H$ is isomorphic to one of the following
\begin{itemize}
\item[(1)] $\field\left[x,y\right]/(x^p,y^p)$,
\item[(2)] $\field\left[x,y\right]/(x^p,y^p-x)$,
\item[(3)] $\field\left[x,y\right]/(x^p-x,y^p-y)$,
\end{itemize}
where the coalgebra structure is given by 
\begin{align}\label{comultiplicationD1P1}
\Delta(x)&=x\otimes 1+1\otimes x,\\
\Delta(y)&=y\otimes 1+1\otimes y+\omega(x).\notag
\end{align}
\end{lemma}
\begin{proof}
By the previous argument, we can find elements $x,y\in H$ with the comultiplications given in \eqref{comultiplicationD1P1}. They generate a Hopf subalgebra of $H$ by Remark \ref{BHC}. Since $H$ has dimension $p^2$, $H$ is generated by $x,y$. It is clear that $[x,y]$ is primitive since
\begin{align*}
\Delta\left(\left[x,y\right]\right)&=\left[\Delta(x),\Delta(y)\right]\\
&=\left[x\otimes 1+1\otimes x,y\otimes 1+1\otimes y+\omega\left(x\right)\right]\\
&=\left[x,y\right]\otimes 1+1\otimes \left[x,y\right].
\end{align*}
In other words, we can write $[x,y]=\lambda x$ for some $\lambda\in \field$, which implies that $[x^n,y]=n\lambda\ x^n$ for any $n\ge 1$. Therefore we can show that
\begin{align}\label{commP2xy}
\left[\omega(x),y\otimes 1+1\otimes y\right]&=\left[\sum_{i=1}^{p-1}\frac{(p-1)!}{i!(p-i)!}\ x^i\otimes x^{p-i}\ ,\ y\otimes 1+1\otimes y\right]\\
&=\sum_{i=1}^{p-1}\frac{(p-1)!}{i!(p-i)!}\ \left([x^i,y]\otimes x^{p-i}+x^i\otimes [x^{p-i}, y]\right)\notag\\
&=\sum_{i=1}^{p-1}\frac{(p-1)!}{i!(p-i)!}\ \left(i\lambda\ x^i\otimes x^{p-i}+x^i\otimes (p-i)\lambda\ x^{p-i}\right)\notag\\
&=\sum_{i=1}^{p-1} \frac{p!}{i!(p-i)!}\lambda\ x^i\otimes x^{p-i}\notag\\
&=0\notag.
\end{align}
Since $\omega(x)^p=\omega(x^p)$, we have
\begin{align}\label{D2PY}
\Delta\left(y^p\right)=\left(y\otimes 1+1\otimes y+\omega(x)\right)^p=y^p\otimes 1+1\otimes y^p+\omega(x^p).
\end{align} 
By Theorem \ref{D1}, we can assume that $x^p=0$ or $x^p=x$. When $x^p=0$, according to the above equation \eqref{D2PY}, $y^p$ is primitive. Then we can write $y^p=\mu x$ for some $\mu\in \field$. Thus $\lambda^p x=x\ \ad(y)^p=[x,y^p]=[x,\mu x]=0$, which implies that $\lambda=0$. By further rescaling of the variables, we can assume $\mu$ to be either one or zero, which yields the first two classes. On the other hand, when $x^p=x$, by \eqref{D2PY} again, $y^p-y$ is primitive. Then we can write $y^p=y+\mu x$ for some $\mu \in \field$. Moreover, $[x,y]=[x^p,y]=\ad(x)^py=0$. After the linear translation $y=y'+\sigma x$ satisfying $\sigma^p=\sigma+\mu$, we have $y'^p=y'$ while $\Delta(y')=y'\otimes 1+1\otimes y'+\omega(x)$. This gives the third class. It remains to show those three Hopf algebras are non-isomorphic. The first two are local with different number of minimal generators and the third one is semisimple. Hence they are non-isomorphic as algebras. This completes the classification.
\end{proof}
Finally, the classification for connected Hopf algebras of dimension $p^2$ follows:
\begin{theorem}\label{D2}
Let $H$ be a connected Hopf algebra of dimension $p^2$. When $\dim \Prim(H)=2$, it is isomorphic to one of the following:
\begin{itemize}
\item[(1)] $\field\left[x,y\right]/\left(x^p,y^p\right)$,
\item[(2)] $\field\left[x,y\right]/\left(x^p-x,y^p\right)$,
\item[(3)] $\field\left[x,y\right]/\left(x^p-y,y^p\right)$,
\item[(4)] $\field\left[x,y\right]/\left(x^p-x,y^p-y\right)$,
\item[(5)] $\field\langle x,y\rangle/\left([x,y]-y,x^p-x,y^p\right)$,
\end{itemize}
where $x,y$ are primitive. 
When $\dim \Prim(H)=1$, it is isomorphic to one of the following:
\begin{itemize}
\item[(6)] $\field\left[x,y\right]/(x^p,y^p)$,
\item[(7)] $\field\left[x,y\right]/(x^p,y^p-x)$,
\item[(8)] $\field\left[x,y\right]/(x^p-x,y^p-y)$,
\end{itemize}
where $\Delta\left(x\right)=x\otimes 1+1\otimes x$ and $\Delta\left(y\right)=y\otimes 1+1\otimes y+\omega(x)$.
\end{theorem}
\begin{proof}
By Proposition \ref{BPH}(6), we know $\dim \Prim(H)\le 2$. If $\dim\Prim(H)=2$, then $H$ is primitively generated and $H\cong u(\mathfrak g)$ for some two-dimensional restricted Lie algebra $\mathfrak g$ by Proposition \ref{BPH}(5). Therefore Proposition \ref{D2Lie} provides the classification. When $\dim \Prim(H)=1$, it is directly from Lemma \ref{D2P1C}. Finally, it is clear that the Hopf algebras given in (1)-(5) are non-isomorphic to the ones given in (6)-(8), since their primitive spaces have different dimension. The Hopf algebras in (1)-(5) are obviously non-isomorphic as algebras. Neither are the ones in (6)-(8). This completes the proof. 
\end{proof}
 
\begin{corollary}\label{localp2}
Let $H$ be a local Hopf algebra of dimension $p^2$. Then H is isomorphic to either $k\left[\xi, \eta\right]\big/ (\xi^p, \eta^p)$ or $k\left [\xi\right] \big/(\xi^{p^2})$ as algebras. When $H\cong k [\xi, \eta] / (\xi^p, \eta^p)$, the coalgebra structure is given by one of the following:
\begin{itemize}
\item[(1)] $\Delta\left(\xi\right)=\xi\otimes 1+1\otimes \xi,\\ \Delta\left(\eta\right)=\eta\otimes 1+1\otimes \eta$,
\item[(2)] $\Delta\left(\xi\right)=\xi\otimes 1+1\otimes \xi+\xi\otimes \xi,\\ \Delta\left(\eta\right)=\eta\otimes 1+1\otimes \eta$,
\item[(3)] $\Delta\left(\xi\right)=\xi\otimes 1+1\otimes \xi,\\ \Delta\left(\eta\right)=\eta\otimes 1+1\otimes \eta+\omega\left(\xi\right)$,
\item[(4)] $\Delta\left(\xi\right)=\xi\otimes 1+1\otimes \xi+\xi\otimes \xi,\\ \Delta\left(\eta\right)=\eta\otimes 1+1\otimes \eta+\eta\otimes \eta$,
\item[(5)] $\Delta\left(\xi\right)=\xi\otimes 1+1\otimes \xi+\xi\otimes \xi,\\ \Delta\left(\eta\right)=\eta\otimes 1+1\otimes \eta+\xi\otimes \eta$.
\end{itemize}
When $H\cong k\left[\xi\right]\big/ (\xi^{p^2})$, the coalgebra structure is given by
\begin{itemize}
\item[(6)] $\Delta\left(\xi\right)=\xi\otimes 1+1\otimes \xi$,
\item[(7)] $\Delta\left(\xi\right)=\xi\otimes 1+1\otimes \xi+\omega\left(\xi^p\right)$,
\item[(8)] $\Delta\left(\xi\right)=\xi\otimes 1+1\otimes \xi+\xi\otimes \xi$.
\end{itemize}
\end{corollary}
\begin{proof}
Denote the dual Hopf algebra of $H$ by $H^*$. By Proposition \ref{BPH}(1), $H^*$ is a connected Hopf algebra of dimension $p^2$. When $\dim \Prim(H^*)=2$, as shown in Theorem \ref{D2}, there are five non-isomorphic classes for $H^*$. By duality, there are also five non-isomorphic classes for $H$. Furthermore, from Proposition \ref{BPH}(4), $\dim J/J^2=\dim \Prim(H^*)=2$, where $J$ is the Jacobson radical of $H$. Notice that $H^*$ is cocommutative. Then $H$ is commutative and we have $H\cong \field[\xi,\eta]/(\xi^p,\eta^p)$ by \cite[Thm. 14.4]{GTM66}. It is easy to check that the coalgebra structures given in $(1)$-$(5)$ are non-isomorphic. The same argument applies to the other case. Theorem \ref{D2} shows that when $\dim \Prim(H^*)=1$, there are three non-isomorphic classes. Since $\dim J/J^2=\dim \Prim(H^*)=1$, $H$ is isomorphic to $\field[\xi]/(\xi^{p^2})$ as algebras. Because those given in $(6)$-$(8)$ are non-isomorphic as coalgebras. They complete the list. 
\end{proof}
\begin{remark}
In fact, the Hopf algebras in Corollary \ref{localp2} (1)-(8) are in one-to-one correspondence with those in Theorem \ref{D2} (1)-(8) via duality. Below, in each case, we describe the generator(s) $\xi,\ \eta$  as linear functional(s) on the basis  $\{x^i\ y^j\, |\, 0\le i,j\le p-1\}$.
\begin{align}
&\xi\left(x^iy^j\right)=
\begin{cases}
1  &  i=1,\ j=0\\
0  & \mbox{otherwise}\\
\end{cases},\quad
\eta\left(x^iy^j\right)=
\begin{cases}
1  &  i=0,\ j=1\\
0  & \mbox{otherwise}\tag{1}\\
\end{cases}\\
&\xi\left(x^iy^j\right)=
\begin{cases}
1  &  i\neq 0,\ j=0\\
0  & \mbox{otherwise}\\
\end{cases},\quad
\eta\left(x^iy^j\right)=
\begin{cases}
1  &  i=0,\ j=1\\
0  & \mbox{otherwise}\tag{2}\\
\end{cases}\\
&\xi\left(x^iy^j\right)=
\begin{cases}
1  &  i=1,\ j=0\\
0  & \mbox{otherwise}\\
\end{cases},\quad
\eta\left(x^iy^j\right)=
\begin{cases}
-1  &  i=0,\ j=1\\
0  & \mbox{otherwise}\tag{3}\\
\end{cases}\\
&\xi\left(x^iy^j\right)=
\begin{cases}
1  &  i\neq 0,\ j=0\\
0  & \mbox{otherwise}\\
\end{cases},\quad
\eta\left(x^iy^j\right)=
\begin{cases}
1  &  i=0,\ j\neq 0\\
0  & \mbox{otherwise}\tag{4}\\
\end{cases}\\
&\xi\left(x^iy^j\right)=
\begin{cases}
1  &  i\neq 0,\ j=0\\
0  & \mbox{otherwise}\\
\end{cases},\quad
\eta\left(x^iy^j\right)=
\begin{cases}
1  &  j=1\\
0  & \mbox{otherwise}\tag{5}\\
\end{cases}\\
&\xi\left(x^iy^j\right)=
\begin{cases}
1  &  i=1,\ j=0\\
0  & \mbox{otherwise}\tag{6-8}
\end{cases}.
\end{align}
\end{remark}
\begin{theorem}\label{centerP1}
Let $H$ be a finite-dimensional connected Hopf algebra with $\dim\Prim(H)=1$. Then the center of $H$ contains $\Prim(H)$. 
\end{theorem}
\begin{proof}
Suppose $\Prim(H)$ is spanned by $x$. By Corollary \ref{FCLH}, $H$ has an increasing sequence of normal Hopf subalgebras: 
\begin{align*}
\field=N_0\subset N_1\subset N_2\subset \cdots \subset N_n=H
\end{align*} 
satisfying $N_1$ is generated by $x$ and $N_{n-1}\subset H$ is normal with $p$-index one. We show by induction on $n$ such that the center of $H$ contains $x$. It is trivial when $n=1$. Assume that $n\ge 2$. Then by Theorem \ref{Cohomologylemma}, we can find some $y\in H\setminus N_{n-1}$ such that $\Delta(y)=y\otimes 1+1\otimes y+u$, where $u\in N_{n-1}\otimes N_{n-1}$, which together with $N_{n-1}$ generate $H$. Apply Theorem \ref{FREENESS} to $N_{n-1}\subset H$, we have $y^p+\lambda\ y+a=0$ for some $\lambda \in \field$ and $a\in N_{n-1}$.   

By induction, $x\in Z(N_{n-1})$. Then it suffices to show $[x,y]=0$. It is easy to check that $[x,y]$ is primitive. Therefore we can write $[x,y]=\mu x$ for some $\mu \in \field$. By rescaling, we can further assume either $x^p=0$ or $x^p=x$. When $x^p=0$, by Theorem \ref{NPLA}, $H$ is local. Then its quotient $H/N_{n-1}^+H$, which is generated by the image of $y$, is local too. Hence the image of $y$ in $H/N_{n-1}^+H$ is nilpotent since it is primitive. Thus in the relation $y^p+\lambda\ y+a=0$, we must have $\lambda=0$ and $y^p+a=0$. A calculation therefore shows that $\mu^px=x(\ad y)^p=[x,y^p]=[x,-a]=0$ which implies that $[x,y]=\mu x=0$. When $x^p=x$, we have $[x,y]=[x^p,y]=(\ad x)^py=0$. This completes the proof.
\end{proof}
\appendix

\section{Restricted Lie algebras}
We state the following technical lemma which is the key to our classification of finite-dimensional connected Hopf algebras.
\begin{lemma}{\cite[P. 186-187]{JA79}}\label{palgebra}
For any associative $\field$-algebra $A$, we have
\begin{align*}
\left(x+y\right)^{p}=x^{p}+y^{p}+\sum_{i=1}^{p-1} s_i\left(x,y\right)
\end{align*}
where $is_i(x,y)$ is the coefficient of $\lambda^{i-1}$ in $x\left(\ad(\lambda x+y)\right)^{p-1}$ and 
\begin{align*}
\left[x^p,y\right]=\left(\ad\ x\right)^p(y)
\end{align*}
for any $x,y\in A$.
\end{lemma}
\begin{definition}\cite[Chapter V Def. 4]{JA79}\label{DefRLA}
A \bf{restricted Lie algebra}\rm\ $\mathfrak{g}$ over $\field$ is a Lie algebra in which there is defined a map $\mathfrak g\to \mathfrak g$, i.e., $x\mapsto x^{[p]}$ such that 
\begin{itemize}
\item[(1)] $\left(\alpha x\right)^{[p]}=\alpha^p x^{[p]}$,
\item[(2)] $\left(x+y\right)^{[p]}=x^{[p]}+y^{[p]}+\sum_{i=1}^{p-1} s_i(x,y)$, where $is_i(x,y)$ is the coefficient of $\lambda^{i-1}$ in $x\left(\ad(\lambda x+y)\right)^{p-1}$,
\item[(3)] $\left[x,y^{[p]}\right]=x\left(\ad\ y\right)^p$,
\end{itemize}
for all $x,y\in \mathfrak{g}$ and $\alpha \in \field$.
\end{definition}
If $\mathfrak{g}$ is restricted and $\mathcal U(\mathfrak g)$ is the usual universal enveloping algebra, let $B$ be the ideal in $\mathcal{U}(\mathfrak g)$ generated by all $x^p-x^{[p]},x\in \mathfrak g$, and define $u(\mathfrak {g})=\mathcal U(\mathfrak{g})/B$. Then $u(\mathfrak g)$ is called the restricted universal enveloping algebra of $\mathfrak g$. A version of the \upshape{PBW} theorem holds for $u(\mathfrak g)$: given a basis for $\mathfrak g$, the ordered monomials in this basis, where the exponent of each basis element is bounded by $p-1$, form a basis for $u(\mathfrak g)$. Consequently if $\dim \mathfrak g=n$, then $\dim u(\mathfrak g)=p^n$. 

Let $\mathfrak g$ be a two-dimensional Lie algebra with basis $\{x,y\}$. There is, up to isomorphism, a unique two-dimensional non-abelian Lie algebra, and we can assume $[x, y] = y$ without loss of generality. The following result about two-dimensional restricted Lie algebras probably is well-known, see, e.g., \cite[Chapter V \S 8]{JA79}. 
\begin{proposition}\label{D2Lie}
Let $\mathfrak{g}$ be a two-dimensional restricted Lie algebra with basis $\{x,y\}$. Then the restricted maps can be classified as follows:
When $\mathfrak g$ is abelian:
\begin{itemize}
\item[(1)] $x^{[p]}=0,y^{[p]}=0$,
\item[(2)] $x^{[p]}=x,y^{[p]}=0$,
\item[(3)] $x^{[p]}=y,y^{[p]}=0$,
\item[(4)] $x^{[p]}=x,y^{[p]}=y$.
\end{itemize}
When $\mathfrak g$ is non-abelian such that $[x,y]=y$:
\begin{itemize}
\item[(5)] $x^{[p]}=x,y^{[p]}=0$.
\end{itemize}
\end{proposition}
\begin{proof}
First suppose $\mathfrak g$ is abelian. Then by \cite[Ex. 19]{JA79}, $\mathfrak g$ can be decomposed into a direct sum $\mathfrak g=\mathfrak g_0\oplus \mathfrak g_1$, where $\mathfrak g_0^{p^n}=0$ for sufficient large $n$ and $\mathfrak g_1^p=\mathfrak g_1$. Define the non-commutative polynomial ring $\Phi=\{\alpha_0+\alpha_1t+\cdots \alpha_nt^n|\alpha_i\in\field\}$, where $t$ is an indeterminate such that $t\alpha=\alpha^pt$. By comments \cite[P. 192]{JA79}, $\mathfrak g_0$ can be viewed as a module over $\Phi$ with $t$ acts on $\mathfrak g_0$ by the restricted map. Hence $\mathfrak g_0$ is annihilated by $t^n$ for $n>>0$. Notice that $\Phi$ is a PID. Thus
\begin{align*}
\mathfrak g_0\cong \bigoplus_i \Phi\big/(t^{n_i})
\end{align*}
as $\Phi$-modules. Suppose $\dim \mathfrak g_1=0$. Then $\mathfrak g_0$ is either isomorphic to the cyclic module of dimension two over $\Phi$, or isomorphic to the direct sum of two copies of the one-dimensional cyclic module over $\Phi$. By applying \cite[Chapter V \S 8 Thm. 13]{JA79} to $\mathfrak g_1$, it is easy to see that the first one gives case $(3)$ and the second one gives case $(1)$. If $\dim \mathfrak g_1=1$, we have case $(2)$. If $\dim \mathfrak g_1=2$, it is case $(4)$. Moreover, they are all non-isomorphic because of the different decompositions and module structures over $\Phi$. When $\mathfrak{g}$ is non-abelian, by the condition $(3)$ of Definition \ref{DefRLA}, we have $[x,x^{[p]}]=[y,y^{[p]}]=[x,y^{[p]}]=0$ and $[x^{[p]},y]=y$. Since $[x,y]=y$, we have $x^{[p]}=x,y^{[p]}=0$. 
\end{proof}

\section{Acknowledgment}
The author would like to greatly acknowledge the invaluable help and guidance from his Ph.D. advisor, Professor J.J. Zhang, and appreciates the help provided by Professors Greenberg, Masuoka and Premet for advice on some proofs. The author is grateful to the referee for his/her careful reading and many useful comments that improved the final version of the paper, and also thankful to Cris Negron and Guangbin Zhuang for their suggestions and comments on the earlier drafts of the paper.

\end{document}